\documentclass{amsart}[12pt]
\usepackage{mathrsfs, amsmath,amssymb}
\usepackage{fullpage}
\usepackage{graphicx}
\newtheorem{teo}{Theorem}

\newtheorem*{remark}{Remark}

\title{Integral representation of the sub-elliptic heat kernel on the complex anti-de Sitter fibration}

\author[F. Baudoin]{Fabrice Baudoin}
\address{Department of Mathematics, University of Connecticut, USA}
\email{fabrice.baudoin@uconn.edu}

\author[N. Demni]{Nizar Demni}
\address{IRMAR, Universit\'e de Rennes 1\\ Campus de
Beaulieu\\ 35042 Rennes cedex\\ France}
\email{nizar.demni@univ-rennes1.fr}

\thanks{F. Baudoin  is supported in part by NSF Grant DMS-1660031}
\begin{document}
\maketitle

\begin{abstract}
We derive an integral representation for the subelliptic heat kernel of the complex anti-de Sitter fibration. Our proof is different from the one used in J. Wang \cite{Wan} since it appeals to the commutativity of the D'Alembertian and of the Laplacian acting on the vertical variable rather than the analytic continuation of the heat semigroup of the real hyperbolic space. Our approach also sheds the light on the connection between the sub-Laplacian of the above fibration and the so-called generalized Maass Laplacian, and on the role played by the odd dimensional real hyperbolic space.
\end{abstract}

\section{Introduction}

The study of explicit expressions for subelliptic heat kernels and related Green functions on Cauchy-Riemann (CR) model spaces is motivated by many fundamental questions in:
\begin{itemize}
\item Analysis, where explicit expressions of Green functions play an important role in the analysis of Sobolev-type inequalities and subelliptic estimates (see for instance \cite{Fol, Je-Lee});
\item Geometry, where explicit heat kernels can be used to compute distances and diameters (see \cite{Bau-Bon2, Bau-Wan1});
\item Probability, where explicit heat kernels are also used to compute limit laws (see \cite{Bau-Wan2}).
\end{itemize}
We refer to the monograph \cite{Calin} and to its bibliography for a general presentation of heat kernels and of their importance in both the elliptic and the subelliptic frameworks.

The aim of this note is to derive an integral representation of the subelliptic heat kernel of the complex anti-de Sitter (hereafter AdS) fibration, which is the CR Sasakian model space of negative curvature. Previous integral representations were obtained in \cite{Bau-Bon} and in \cite{Wan} where an analytic continuation of the heat semigroup mysteriously led  to the real hyperbolic space having the same real dimension as the AdS. Moreover, in these two papers, the question whether the commutativity of the D'Alembertian of the AdS space and the Laplacian acting on the fibers (circles) may be used to derive such an integral representation was also raised.
The present note answers positively this question and simplifies the proofs in \cite{Bau-Bon, Wan}. It also provides an evidence for the occurrence of the heat kernel of the odd-dimensional real hyperbolic space. 

Finally, let us point out that the connection we discover between the sub-Laplacian of the AdS fibration and the generalized Maass Laplacian is an important observation that extends to the negatively-curved setting the known one between the sub-Laplacian of the Heisenberg group and the magnetic Laplacian in the flat complex space (see e.g. \cite{Koc-Ric}). 


\section{The AdS fibration}

In $\mathbb{C}^{n+1}, n \geq 1$, we consider the signed quadratic form: 
\begin{equation*}
||(z_1, \dots, z_{n+1})||_H = \sum_{i=1}^n |z_i|^2 - |z_{n+1}|^2. 
\end{equation*}
Then, the AdS space is the quadric defined by 
\begin{equation*}
\{z = (z_1, \dots, z_{n+1}) \in \mathbb{C}^{n+1}, ||z||_H = -1\},
\end{equation*}
and the circle group $U(1)$ acts on it by $z \mapsto ze^{i\theta}$. This action gives rise to the AdS fibration over the complex hyperbolic (Bergman) ball $H_n(\mathbb{C})$ with $U(1)$-fibers. The identification of the base space is more transparent if we consider the projective model of the AdS space:
\begin{equation*}
\{z = (z_1, \dots, z_{n+1}) \in \mathbb{C}^{n+1}, ||z||_H < 0\},
\end{equation*}
and use the inhomogeneous coordinates $w_i := z_i/z_{n+1}, i =1, \dots, n$ which obviously satisfies $||w|| < 1$ (here, $||\cdot||$ is the Euclidean norm of $\mathbb{C}^n$).  The corresponding sub-Laplacian is then the horizontal lift to the AdS space of the Laplace-Beltrami operator (for the standard Bergman metric) on $H_n(\mathbb{C})$.  It is important to observe that the submersion $AdS \to H_n(\mathbb{C})$ is only semi-Riemannian in the sense of \cite{Bad-Ian}, but the sub-Laplacian is a subelliptic operator.

We aim to prove the following theorem, whose exact notations will be explained in the text:

\begin{teo}\label{main}
The heat kernel of the sub-Laplacian on the AdS space is given by

\[
 \frac{1}{\sqrt{4\pi t}} \sum_{k \in \mathbb{Z}} \int_{\mathbb{R}} e^{(u-i\theta - 2ik\pi)^2/(4t)} q_t( \cosh(u)\cosh(d(0,y))) du
\]
where $q_t$ is the Riemannian heat kernel on the $2n+1$-dimensional real hyperbolic space ${\bf H}^{2n+1}$.
\end{teo}

The theorem in this form is originally due to J. Wang \cite{ Wan}. However, the  role played by the $2n+1$-dimensional real hyperbolic space in the study of the Ads fibration is not apparent in \cite{Wan} and one of the features of our new proof is to show where it comes from (see Remark \ref{hyperbolic} below).

\section{The complex AdS d'Alembertian and the generalized Maass Laplacian}

Let $\theta \in \mathbb{R}/2\pi \mathbb{Z}$ be the coordinate fiber on $U(1)$, $(w_1, \dots, w_n) \in H_n(\mathbb{C})$ and write simply: 
\begin{equation*}
\partial_i := \partial_{w_i}, \quad \overline{\partial}_i := \partial_{\overline{w_i}}.
\end{equation*}

The D'Alembertian of the AdS space induced by the flat $(2n,2)$-signed metric of $\mathbb{R}^{2n,2}$ reads (see \cite{Wan}, p.9): 
\begin{equation*}
\square = 4(1-||w||^2)\left\{\sum_{i,j=1}^n \left(\delta_{i,j} - w_i \overline{w_j}\right)\partial_i\overline{\partial_j}+ \frac{i}{2}\sum_{i=1}^n\left(w_i\partial_i - \overline{w_i}\overline{\partial_i} \right)\partial_{\theta} - \frac{1}{4}\partial^2_{\theta}\right\}.
\end{equation*}
Acting on functions $g: AdS \rightarrow \mathbb{C}$ of the form $g(w_1, \dots, w_n, \theta) = e^{-2ik\theta}f(w_1, \dots, w_n), 2k \in \mathbb{Z}$, this operator reduces to
\begin{equation*}
\square(g)(w_1, \dots, w_n, \theta) = e^{-2ik\theta} D_k(f)(w_1, \dots, w_n) 
\end{equation*}
where 
\begin{equation*}
D_k = 4(1-||w||^2)\left\{\sum_{i,j=1}^n \left(\delta_{i,j} - w_i \overline{w_j}\right)\partial_i\overline{\partial_j}+ k \sum_{i=1}^n\left(w_i\partial_i - \overline{w_i}\overline{\partial_i} \right) + k^2\right\}. 
\end{equation*}
The operator $D_k$ for $k \neq 0$ is known as the generalized Maass operator since for $n=1$, it reduces after performing a weighted Cayley transform to the operator introduced by H. Maass in his investigations of non-analytic automorphic forms on the 
Poincar\'e half-plane (see e.g \cite{Aya-Int}, Remark 2.1): 
\begin{equation*}
\Delta_k = y^2(\partial_x^2 + \partial_y^2) -2ikx\partial_x, \quad x \in \mathbb{R}, y > 0.
\end{equation*}
 The latter is also known as the magnetic Schr\"odinger operator with uniform magnetic field of strength $k$. Moreover, $D_k$ is a Bochner Laplacian:  
\begin{equation*}
D_k =  (-d+ik \alpha)^{\star}(d+ik\theta) + 4k^2,
\end{equation*}
where $d$ is the complex exterior derivative, $\star$ stands the adjoint operator induced by the Bergman metric on differential forms and $\theta$ is the one-form (\cite{Wel}): 
\begin{equation*}
\alpha = -i(\partial_w - \partial_{\overline{w}})\log(1-||w||^2) = \frac{i}{(1-||w||^2)}\sum_{j=1}^n \left(\overline{w_j} dw_j - w_j d\overline{w_j}\right), \quad \partial_{w} := \sum_{i=1}^n \partial_i.
\end{equation*}
It is worth noting that the exterior derivative $d\alpha$ is the K\"ahler form of $H_n(\mathbb{C})$ and that its pull back by the Pseudo-Riemannian submersion map 
\begin{equation*}
(z_1, \dots, z_{n+1}) \rightarrow (w_1,\dots, w_n)
\end{equation*}
 is the contact CR form on the AdS space: 
\begin{equation*}
\eta = -d\theta + \alpha = \Im \left(\sum_{i=1}^n\overline{z_i}dz_i - \overline{z_{n+1}}dz_{n+1}\right),  
\end{equation*}
inducing the horizontal distribution in AdS. That is why the generalized Maass Laplacian is connected to the horizontal lift of the Laplace-Beltrami operator on $H_n(\mathbb{C})$ and in turn to $\square$. A similar connection is already known between the magnetic  
Laplacian in $\mathbb{C}^n$ with uniform field and the sub-Laplacian of the Heisenberg group (\cite{Koc-Ric}), and holds true as well for the Hopf fibration (compare Prop. 2.1, (iv) in \cite{Haf-Int} and Prop. 2.1 in \cite{Bau-Wan}). 

The key ingredient in our proof of Theorem \ref{main} is the following result (see \cite{Aya-Int}, Theorem 2.2, (i)): 
\begin{teo}\label{main1}
For any $t > 0, n \geq 1$, let $v_{t,n,k}$ be the heat kernel of $D_k$ with respect to the Bergman volume measure: 
\begin{equation*}
\frac{dy}{(1-||y||^2)^{n+1}}.
\end{equation*}
Then, 
\begin{equation}\label{IntRe}
v_{t,n,k}(w,y) = \frac{e^{-n^2t}}{(2\pi)^n\sqrt{4\pi t}}  \left(\frac{1-\overline{\langle w,y \rangle}}{1-\langle w, y\rangle}\right)^k 
 \int_{\mathbb{R}} dx \sinh(x) N_k(x,w,y) \left(-\frac{1}{\sinh(x)}\frac{d}{dx}\right)^n e^{-x^2/(4t)},
\end{equation}
where 
\begin{equation*}
N_k(x,w,y) = \frac{1}{\sqrt{\cosh^2(x) - \cosh^2(d(w,y))_+}} {}_2F_1\left(-2k, 2k, \frac{1}{2}; \frac{\cosh(d(w,y)) - \cosh(x)}{2\cosh(d(w,y))}\right),
\end{equation*}
${}_2F_1$ is the Gauss hypergeometric function (\cite{AAR}), and $d(w,y)$ is the hyperbolic distance: 
\begin{equation*}
\cosh^2(d(w,y)) = \frac{|1-\langle w, y\rangle|^2}{(1-||y||^2)(1-||w||^2)}. 
\end{equation*}
\end{teo}
In \cite{Aya-Int}, the authors did not write a detailed proof of this Theorem and rather referred to an unpublished paper. For that reason and for the reader's convenience, we provide below the missing details. 
\begin{proof}
Consider the following wave equation 
\begin{equation*}
\partial_t^2 u(t,w) = (D_k + n^2) u(t,w), \quad t \in \mathbb{R}, w \in H_n(\mathbb{C}), 
\end{equation*}
with intial values 
\begin{equation*}
u(0,w) = 0, \quad \partial_t u(t,w) = f(w),
\end{equation*}
where $f \in C_0^{\infty}(H_n(\mathbb{C}))$. In \cite{Int-OM}, the authors obtained the solution of the wave equation for an operator denoted $\Delta_{\alpha \beta}$ there which reduces to $D_k + n^2$ when $\alpha = k = -\beta$. From Theorem 2 there, it follows that 
\begin{equation*}
u(t,w) =  \frac{1}{(2\pi)^n} \left(-\frac{1}{\sinh(t)}\frac{d}{dt}\right)^{n-1} \int_{H_n(\mathbb{C})} N_k(t,w,y)f(y) \frac{dy}{(1-||y||^2)^{n+1}}.  
\end{equation*}
Now, the wave equation: 
\begin{equation*}
\partial_t^2 h(t,w) = (D_k +n^2) h(t,w), \quad t \in \mathbb{R}, w \in H_n(\mathbb{C}), 
\end{equation*}
with intial values 
\begin{equation*}
h(0,w) = f(w), \quad \partial_t h(t,w) = 0,
\end{equation*}
is satisfied by 
\begin{equation*}
h(t,w) = \partial_t u(t,w) = -\frac{\sinh(t)}{(2\pi)^n} \left(-\frac{1}{\sinh(t)}\frac{d}{dt}\right)^{n} \int_{H_n(\mathbb{C})} N_k(t,w,y) f(y)\frac{dy}{(1-||y||^2)^{n+1}}.
\end{equation*}
Indeed, 
\begin{equation*}
u(t,w) = \frac{\sin(t\sqrt{-D_k-n^2})}{\sqrt{-D_k-n^2}}(f)(w)
\end{equation*}
while 
\begin{equation*}
h(t,w) = \cos(t\sqrt{-D_k-n^2})(f)(w),
\end{equation*}
where these formulas are defined by means of the spectral Theorem. Finally, the following formula holds in the weak sense (see e.g. \cite{Gri-Nog}): for any $t > 0$ and any non positive selfadjoint operator $L$,
\begin{equation}\label{FormSpec}
e^{tL} = \frac{1}{\sqrt{4\pi t}}\int_{\mathbb{R}} e^{-x^2/4} \cos(x\sqrt{-L}) dx. 
\end{equation}
Taking $L = D_k + n^2$, plugging the expression of $h(t,z)$ in the right hand side of \eqref{FormSpec} and performing successive integrations by parts there, we obtain the expression of the heat kernel for $D_k$ displayed in the Theorem.   
\end{proof}
\begin{remark}\label{hyperbolic}
Note that 
\begin{equation*}
\frac{e^{-n^2t}}{(2\pi)^n\sqrt{4\pi t}}\left(-\frac{1}{\sinh(x)}\frac{d}{dx}\right)^n e^{-x^2/(4t)} 
\end{equation*} 
is the heat kernel on the real hyperbolic space ${\bf H}^{2n+1}$ of dimension $2n+1$ (see e.g. \cite{Gri-Nog}). It arose in \cite{Bau-Bon} and in \cite{Wan} from the analytic continuation of the AdS space and was denoted there by $q_t(\cosh(x))$ where $x$ is interpreted as the geodesic distance in ${\bf H}^{2n+1}$ from the north pole. Moreover, if $k=0$, $v_{t,n,0}$ reduces to the heat kernel of the complex hyperbolic space and the expression displayed in the Theorem coincides with Theorem 3.1, (i) in \cite{OM}. 
\end{remark}

With the help of Theorem \ref{main1}, we are ready to prove our main result. 

\begin{proof}[Proof of the Theorem \ref{main}]

Let $p_{t,n} (w,y,\theta)$ denote the heat kernel of $\square$ with respect to the reference measure: 
\begin{equation*}
\frac{dy}{(1-||y||^2)^{n+1}} \frac{d\theta}{2\pi}.
\end{equation*}

This is smooth in the fiber coordinate $\theta \in \mathbb{R}/2\pi \mathbb{Z}$, and as such it has the Fourier expansion: 
\begin{equation*}
p_{t, n}(w,y,\theta) = \sum_{k \in \mathbb{Z}} v_{t,n,k/2}(w,y)e^{-ik\theta}, \quad w,y \in H_n(\mathbb{C}).
\end{equation*}
Applying the operator $e^{t\partial_{\theta}^2}$ to $p_{t,n,k}(w,y,\cdot)$, we obtain the following series for the subelliptic heat kernel of the complex AdS space:
\begin{align*}
e^{t\partial_{\theta}^2}p_{t,n}(w,y,\cdot)(\theta) & =  \int_{\mathbb{R}} \sum_{k \in \mathbb{Z}} v_{t,n,k/2}(w,y)e^{-ik(r+\theta)} e^{-r^2/(4t)}\frac{dr}{\sqrt{4\pi t}}
\\& = \sum_{k \in \mathbb{Z}} v_{t,n,k/2}(w,y)e^{-ik\theta} e^{-tk^2}.
\end{align*}

Now, perform the variable change $\cosh(x) = \cosh(u)\cosh(d(w,y)), x > 0, u >0,$ in \eqref{IntRe} and take $w=0$ in order to derive: 
\begin{equation*}
e^{t\partial_{\theta}^2}p_t(0,y,\cdot)(\theta) = 2\sum_{k \in \mathbb{Z}}e^{-ik\theta} e^{-tk^2} \int_{0}^{\infty} du  \,\, q_t( \cosh(u)\cosh(d(0,y))) {}_2F_1\left(-k, k, \frac{1}{2}; \frac{1 - \cosh(u)}{2}\right).
\end{equation*}
Next, notice that 
\begin{equation*}
{}_2F_1\left(-k, k, \frac{1}{2}; \frac{1 - \cosh(u)}{2}\right) = \cosh(ku),
\end{equation*}
therefore 
\begin{equation*}
e^{t\partial_{\theta}^2}p_t(0,y,\cdot)(\theta) = \sum_{k \in \mathbb{Z}}e^{-ik\theta} e^{-tk^2} \int_{\mathbb{R}} du  \,\, q_t( \cosh(u)\cosh(d(0,y))) e^{ku}.
\end{equation*}
On the other hand, the Poisson summation formula: 
\begin{equation*}
\sqrt{x}\sum_{k \in \mathbb{Z}} e^{-\pi x(k+s)^2} = \sum_{k \in \mathbb{Z}} e^{-\pi k^2/x} e^{2i\pi k s}, \quad x > 0, s \in \mathbb{R}, 
\end{equation*}
entails
\begin{align*}
\sum_{k \in \mathbb{Z}} e^{(u-i\theta - 2ik\pi)^2/(4t)} & = e^{(u-i\theta)^2/(4t)} \sum_{k \in \mathbb{Z}} e^{-k^2\pi^2/t} e^{ik\pi(u-i\theta)/t}  
\\& = \sqrt{\frac{t}{\pi}}  e^{(u-i\theta)^2/(4t)}\sum_{k \in \mathbb{Z}}   e^{-t((u-i\theta)/(2t) + k)^2} 
\\& =  \sqrt{\frac{t}{\pi}}  \sum_{k \in \mathbb{Z}}   e^{-tk^2} e^{k(u-i\theta)}. 
\end{align*}
As a result, Fubini Theorem entails
\begin{align*}
\frac{1}{\sqrt{4\pi t}} \sum_{k \in \mathbb{Z}} \int_{\mathbb{R}} e^{(u-i\theta - 2ik\pi)^2/(4t)} q_t( \cosh(u)\cosh(d(0,y))) du & = \int_{\mathbb{R}}\sum_{k \in \mathbb{Z}}   e^{-tk^2} e^{k(u-i\theta)} q_t( \cosh(u)\cosh(d(0,y)))  \frac{du}{2\pi}
\\& = \frac{1}{2\pi} e^{t\partial_{\theta}^2}p_t(0,y,\cdot)(\theta),
\end{align*}
as desired (note that the factor $1/(2\pi)$ is due to our normalization of the reference measure on the fiber). 
\end{proof}

\end{document}